\documentclass[12pt,twoside]{amsart}
\usepackage{amssymb}
\usepackage{graphicx}
\usepackage[margin=1in]{geometry}

\newcommand{\beas}{\begin{eqnarray*}}
\newcommand{\eeas}{\end{eqnarray*}}
\newcommand{\bea}{\begin{eqnarray}}
\newcommand{\eea}{\end{eqnarray}}
\newcommand{\beq}{\begin{equation}}
\newcommand{\eeq}{\end{equation}}
\newcommand{\ben}{\begin{enumerate}}
\newcommand{\een}{\end{enumerate}}

\newtheorem{theorem}{Theorem}[section]
\newtheorem{lemma}[theorem]{Lemma}
\newtheorem{proposition}[theorem]{Proposition}

\newtheorem{conjecture}[theorem]{Conjecture}

\theoremstyle{definition}
\newtheorem{remark}[theorem]{Remark}

\usepackage{latexsym}
\usepackage{color,hyperref}
\definecolor{darkblue}{rgb}{0,0,0.6}
\hypersetup{colorlinks,breaklinks,linkcolor=darkblue,urlcolor=darkblue,anchorcolor=darkblue,citecolor=darkblue}

\title[Unimodality of partitions with distinct parts inside Ferrers shapes]{Unimodality of partitions with distinct parts\\inside Ferrers shapes}

\author[Richard P. Stanley]{Richard P. Stanley${}^{\ast }$} \address{Department of Mathematics\\ MIT\\ Cambridge, MA 02139-4307}
\email{rstan@math.mit.edu}

\author[Fabrizio Zanello]{Fabrizio Zanello${}^{\ast \ast}$} \address{Department of Mathematical  Sciences\\ Michigan
  Tech\\ Houghton, MI  49931-1295} 
\email{zanello@mtu.edu}

\thanks{2010 {\em Mathematics Subject Classification.} Primary: 05A17;
  Secondary: 05A15, 05A16, 05A19.\\\indent 
{\em Key words and phrases.} Integer partition; shifted Ferrers
diagram; $q$-binomial coefficient; partition with distinct parts;
unimodality; $q$-analog; generating function; bijective proof.\\\indent
${}^*$This author's contribution is based upon work supported by the National
    Science Foundation under Grant No.~DMS-1068625.\\\indent
    ${}^*{}^*$This author is partially supported by a
Simons Foundation grant (\#274577).}%Corresponding author. Phone: (001) 906-487-2068; Fax: (001) 906-487-3133}  

\begin{document}

\begin{abstract} We investigate the rank-generating function
  $F_{\lambda}$ of the poset of partitions contained inside a given
  shifted Ferrers shape $\lambda$. When $\lambda $ has four parts, we
  show that $F_{\lambda}$ is unimodal when $\lambda =\langle
  n,n-1,n-2,n-3 \rangle$, for any $n\ge 4$, and that unimodality fails
  for the doubly-indexed, infinite family of partitions of the form
  $\lambda=\langle n,n-t,n-2t,n-3t \rangle$, for any given $t\ge 2$
  and $n$ large enough with respect to $t$.

When $\lambda $ has $b\le 3$ parts, we show that our rank-generating
functions $F_{\lambda}$ are all unimodal. However, the situation
remains mostly obscure for $b\ge 5$. In general, the type of results
that we obtain present some remarkable similarities with those of the
1990 paper of D. Stanton, who considered the case of partitions inside
ordinary (straight) Ferrers shapes.

Along the way, we also determine some interesting $q$-analogs of the
binomial coefficients, which in certain instances we conjecture to be
unimodal. We state several other conjectures throughout this note, in
the hopes to stimulate further work in this area. In
particular, one of these will attempt to place into a much broader
context the unimodality of the posets $M(n)$ of staircase partitions,
for which determining a combinatorial proof remains an outstanding
open problem.
\end{abstract}

\maketitle

\section{Introduction}

A classical result in combinatorics is the unimodality of the $q$-binomial
coefficient (or Gaussian polynomial) $\binom{n+b}{b}_q$, which is the
rank-generating function of the poset of integer partitions having at
most $b$ parts and whose largest part is at most $n$, denoted $L(b,n)$
(see e.g. \cite{Pr1,St1980,St5,Sy}, and of course K. O'Hara's
celebrated combinatorial proof \cite{Oh,Zei}). In other words, the
coefficients of $\binom{n+b}{b}_q$ are \emph{unimodal}, i.e., they do
not increase strictly after a strict decrease.

Recall that a nonincreasing sequence
$\lambda=(\lambda_1,\dots,\lambda_b)$ of positive integers is  a
\emph{partition} of $N$ if $\sum_{i=1}^b\lambda_i=N$. The $\lambda_i$
are the \emph{parts} of $\lambda $, and the index $b$  is  its
\emph{length}.  A partition $\lambda$ can be represented geometrically
by its \emph{Ferrers diagram}, which is  a collection of cells,
arranged in left-justified rows, whose $i$th row contains exactly
$\lambda_i$ cells. With a slight abuse of notation, we will sometimes
also denote by $\lambda $ the Ferrers diagram of the partition
$\lambda $.  

For some useful introductions and basic results of partition
theory, we refer the reader to \cite{And,AE,Pak}, Section I.1 of
\cite{Ma}, and Section 1.8 of \cite{St0}. For any other standard
combinatorial definition, we refer to \cite{St0}. 

The unimodality of $\binom{n+b}{b}_q$ can  be rephrased in terms
of Ferrers diagrams, by saying that the {rank-generating function} of
the {poset} of partitions whose Ferrers diagrams are contained inside
a $b\times n$ rectangle, namely $G_{\lambda}$, where
$\lambda=(\lambda_1=n,\lambda_2=n,\dots,\lambda_b=n)$, is unimodal. In
his 1990 paper \cite{stanton}, D. Stanton studied the rank-generating
function $G_{\lambda}$ of partitions contained inside other Ferrers
shapes $\lambda$. Not surprisingly, $G_{\lambda} $ can be nonunimodal
for certain $\lambda$, the smallest of which turned out to be
$\lambda=(8,8,4,4)$. Stanton was also able to determine infinitely
many nonunimodal partitions $\lambda$ with $b=4$ parts, while he
proved that unimodality always holds when $b\le 3$. He also showed
that nonunimodal partitions exist for $b=6$, whereas all examples known
to date when $b=5$ or $b\ge 7$ are unimodal. 

A well-known variant of the Ferrers diagram of a partition is the
\emph{shifted Ferrers diagram} of a partition $\lambda$ with distinct
parts. Such diagrams have $\lambda_i$ cells in row $i$ as before, but
now each row is indented one cell to the right of the previous row.
The goal of this note is to study the rank-generating function
$F_{\lambda}$ of the poset of partitions $\mu$ contained inside a
shifted Ferrers shape $\lambda$. Equivalently,
$\mu$ is a partition \emph{with distinct parts} contained in an
ordinary (straight) Ferrers shape $\lambda$. We write $\langle
\lambda_1,\lambda_2,\dots,\lambda_n\rangle$ for a partition $\lambda$
with distinct parts, having at most $n$ parts, regarded as a shifted
diagram. Note that 0 is not considered a part, so $\lambda=\langle
4,2,1,0,0\rangle$ is a partition with distinct parts.
For instance, the partitions contained inside $\lambda=\langle
4,2,1\rangle$ 
are: $\emptyset$ (the empty partition); $\langle 1\rangle$ partitioning 1; $\langle 2\rangle$
partitioning 2, $\langle 3\rangle$ and $\langle 2,1\rangle$ partitioning 3; $\langle 4\rangle$ and $\langle 3,1\rangle$
partitioning 4; $\langle 4,1\rangle$ and $\langle 3,2\rangle$ partitioning 5; $\langle 4,2\rangle$ and
$\langle 3,2,1\rangle$ partitioning 6; and $\langle 4,2,1\rangle$ partitioning 7. Thus,  
$$F_{\langle 4,2,1\rangle}(q)=1+q+q^2+2q^3+2q^4+2q^5+2q^6+q^7.$$

While Stanton's work was in part motivated by the interest of the
unimodality of rectangular Ferrers shapes, the corresponding prototype
of partition in our situation is the ``shifted staircase partition''
$\lambda=\langle b,b-1,\dots,2,1\rangle$. The  poset of partitions with distinct
parts that it generates is often referred to as $M(b)$. It is a
standard exercise to show that $F_{\lambda}(q)=\prod_{i=1}^b
(1+q^i)$. The unimodality of this polynomial, which was essentially
first proved by E.B. Dynkin \cite{D1,D2} (see also \cite{Pr1,St5}), is
also closely related to the famous Erd\"os-Moser conjecture, solved by
the first author in \cite{St5}. Notice, however, that the simplest
proof known to date of the unimodality of the staircase partition uses
a linear algebra argument \cite{Pr1}; it remains an outstanding open
problem in combinatorics to determine a constructive proof. 

Though our situation is obviously essentially different from that of
Stanton --- for instance, it is easy to see that our rank-generating
functions $F_{\lambda}$ are never symmetric if $\lambda$ has at least
two parts, with the only exception of the staircase partitions ---
some of our results for distinct parts will show a remarkable
similarity to the case of arbitrary partitions. In  this paper we will
mostly focus on partitions $\lambda$ whose parts are in arithmetic
progression, even though, similarly to what was done in
\cite{stanton}, it is possible to naturally extend some results or
conjectures to partitions having distinct parts that lie within
certain intervals. 

In the next section, we will consider the case when $\lambda$ has four
parts. First, we show that $F_{\lambda}$ is unimodal for all
``truncated staircases'' $\lambda=\langle
n,n-1,n-2,n-3\rangle$. Notice that, unlike in many other instances of
nontrivial unimodality results in combinatorics, in this case
$F_{\lambda}$ is never symmetric (for $n>4$) nor, as it will be clear
from the proof, log-concave.

Our second main result is the existence of a doubly-indexed, infinite
family of nonunimodal rank-generating functions $F_{\lambda}$. Namely,
we will show that if $\lambda=\langle n,n-t,n-2t,n-3t\rangle$, where
$t\ge 2$, then $F_{\lambda}$ is always nonunimodal whenever $n$ is
large enough with respect to $t$ (the least such $n$ can be
computed effectively). For $t=2$, as we will see in the subsequent
section, the rank-generating function $F_{\lambda}$ turns out to be a
$q$-analog of the binomial coefficient $\binom{n+1}{4}$. We will
briefly discuss the meaning of these new $q$-analogs
$\binom{a}{b}^q$. Interestingly, even though, unlike the $q$-binomial
coefficient, in general they can be nonunimodal, we will conjecture
unimodality for our \emph{central} $q$-analogs of the binomial coefficients,
$\binom{2n+1}{n}^q$ and $\binom{2n}{n}^q$.

Similarly to Stanton's situation, we will  show that $F_{\lambda}$
is unimodal for any partition $\lambda$ with at most $b=3$ parts (in
fact, we will rely on Stanton's theorem to give a relatively quick
proof of our result). Moreover, again like Stanton, we are unaware of
the existence of any nonunimodal rank-generating function
$F_{\lambda}$ when $b=5$ or $b\ge 7$, and will provide examples of
nonunimodal $F_{\lambda}$ for $b=6$ that we have not been able to
place into any infinite family. 

Next we conjecture the unimodality of all
partitions $\lambda$ having parts in arithmetic progression that begin
with the smallest possible positive residue. This conjecture, if true,
would place the still little understood unimodality of the staircase
partitions into a much broader context.  Other conjectures are
given throughout this note. 

\section{Partitions of length four}

In this section, we study the rank-generating functions $F_{\lambda}$
of partitions $\lambda$ of length four. We focus in our statements on
partitions whose parts are in arithmetic progression, which is the
most interesting case; i.e., we consider $\lambda$ to be of the form
$\lambda=\langle n,n-t,n-2t,n-3t\rangle$. Notice, however, that certain arguments
could naturally be applied to partitions whose parts lie inside
suitable intervals, similarly to some of the cases studied by Stanton
\cite{stanton}. 

Our first main result of this section is that if
$\lambda=\langle n,n-1,n-2,n-3\rangle$, then $F_{\lambda}$ is unimodal, for any
$n\ge 4$. In contrast, our second  result will show that the
doubly-indexed, infinite family of partitions
$\lambda=\langle n,n-t,n-2t,n-3t\rangle$ are nonunimodal, for any given $t\ge 2$
and $n$  large enough with respect to $t$. The proofs of both results
will be mostly combinatorial, and rely in part on the following
elegant properties of the coefficients of the $q$-binomial coefficients
$\binom{a+4}{4}_q$, which are  of independent interest. 

\begin{lemma}\label{fac}
Let $\binom{a+4}{4}_q=\sum_{i=0}^{4a}d_{a,i}q^i$. Define
$f(a,c)=d_{a,2a-c}-d_{a,2a-c-1}$ for $c\ge 0$, and $f(a,c)=0$ for $c<0$. (Hence, $f(a,c)=0$ for $c>2a$.) We have: 
\begin{enumerate}
\item[(a)] $$\sum_{a,c}f(a,c)q^at^c= \frac{1}{(1-q^2)(1-q^3)(1-qt^2)}+
     \frac{q^2t^2}{(1-q^2)(1-q^3)(1-qt^2)(1-qt)};$$ 
\item[(b)] $f(a,c)\ge 0$ for all $a$ and $c$. Moreover, if $a\ge 2$ and
  $0\le c\le 2a$, then equality holds if and only if $c=1$, $c=2a-1$,
  or $(a,c)=(4,3)$; 
\item[(c)] $f(a,0)=\lfloor (a+3\delta)/6 \rfloor$, where $\delta=1$ if
  $a$ is odd, and $\delta=2$ if $a$ is even. In particular, $f(a,0)$
  goes to infinity when $a$ goes to infinity.
\end{enumerate}
\end{lemma}

\begin{proof} (a) See Theorem 2.2 of \cite{SZ}.

(b) That $f(a,c)\ge 0$ for all $a$ and $c$ is obvious from part (a) or
  also from the fact that the $q$-binomial coefficient
  $\binom{a+4}{4}_q$ is unimodal and therefore nondecreasing up to degree $2a$. As for the second part of the statement, that the coefficients of $q^at$, $q^at^{2a-1}$ and $q^4t^3$ are
  0 for all $a$ is easy to check directly. Proving the converse
  implication requires some careful but entirely standard analysis, so
  we will omit the details.

(c) From part (a), we immediately have that the generating function
  for $f(a,0)$ is  
$$\sum_{a\ge 0}f(a,0)q^a=\frac{1}{(1-q^2)(1-q^3)}.$$

In other words, $f(a,0)$ counts the number of partitions of $a$ whose
parts can only assume the values 2 and 3. That their number now is the
one in the statement is a simple exercise that we leave to the
reader. This completes the proof of the lemma.
\end{proof}

\begin{remark} 
It is possible to give an entirely combinatorial proof of parts (b)
and (c) of Lemma~\ref{fac}, using D. West's symmetric chain
decomposition for the poset $L(4,a)$, whose rank-generating function
is of course $\binom{a+4}{4}_q$ (see \cite{west} for all
details). However, the  argument would be less elegant and require
significantly more work than using part (a) of the lemma.  

On the other hand, the portions of the statement that will later
suffice to show the nonunimodality of $\lambda=\langle
n,n-t,n-2t,n-3t\rangle$ for 
$t\ge 2$ and $n$ large --- namely that $f(a,1)=0$ for all $a$, and
that $f(a,0)$ goes to infinity when $a$ goes to infinity --- are easy and
interesting to show using West's result. Indeed, it is clear that in
his decomposition of $L(4,a)$  there exist no symmetric chains of
cardinality three (one should only check that the cardinality of the
chains $D_{i,j}$ defined at the middle of page 13 of \cite{west}
cannot equal 3, by how the indexes $i$ and $j$ are defined for the new
chains on page 7). This immediately gives that $f(a,1)=0$, since,
clearly, $f(a,c)= 0$ if and only if  in West's construction there
exist no symmetric chains of cardinality $2c+1$.   

In order to show that $f(a,0)$ goes to infinity, notice  that West's
proof implies that $f(a,0)$ is nondecreasing, since in the inductive
step he makes an injection between the chains constructed for $a-1$
and those for $a$. Thus, in the formula for the cardinality of
$c_{i,j}$ at the middle of page 13 of \cite{west}, one can for
instance  choose $a$ to be a multiple of 6, $i=a/3$, and $j=0$. This
easily shows, for these values of $a$, the existence of an extra chain of cardinality 1 that does not come from $a-1$, which suffices to make $f(a,0)$  go to infinity. (In fact, a little more
work proves in this fashion all of part (c) of Lemma \ref{fac}.) 
\end{remark}

\begin{remark}
We thank one of the anonymous referees for pointing out to us an alternative, combinatorial proof of part (a) of Lemma~\ref{fac}, that we sketch below.

It can be shown that K. O'Hara's proof of the unimodality of $q$-binomial
coefficients \cite{Oh,Zei} implies the following recursion for ${a+4 \choose 4}_q$:
$$ {a+4 \choose 4}_q = {a+1 \choose 4}_q + \sum_{0 \leq j < a/2} q^{2j} {2j+1 \choose 1}_q {4a-6j + 1 \choose 1}_q + \frac{(1+(-1)^a) q^a }{2} {a+2 \choose 2}_q .$$

From this, it follows that if we let $F(q,t)$ denote the generating function in Lemma~\ref{fac}, (a), then the coefficient of $q^k$ in $(1-q^3)F(q,t)$ is given by:
$$  \frac{(1+(-1)^k)}{2} \sum_{0 \leq j \leq k/2} t^{2j}  + \sum_{0 \leq j < k/2} t^{2k-4j} \sum_{i = 0}^{2j} t^i. $$

Finally, by standard algebraic manipulations, one can  see that the last displayed formula is equal to the coefficient of $q^k$ in
$$\frac{1-qt+q^2t^2}{(1-q^2)(1-qt)(1-qt^2)},$$
which easily completes the proof of Lemma~\ref{fac}, (a).
\end{remark}

\begin{theorem}\label{uni}
Let $\lambda=\langle n,n-1,n-2,n-3\rangle$, where $n\ge 4$. Then the
rank-generating function $F_{\lambda}$ is unimodal. 
\end{theorem}

\begin{proof}
In the case when the  partitions lying inside
$\lambda=\langle n,n-1,n-2,n-3\rangle$ have  three or four parts, by removing the staircase
$\langle 3,2,1\rangle$ from $\lambda$, it is easy to see that such partitions are in
bijection with the \emph{arbitrary} partitions contained inside a
$4\times (n-3)$ rectangle, whose rank-generating function is
$\binom{n+1}{4}_q$. In a similar fashion, when the  partitions lying inside
$\lambda=\langle n,n-1,n-2,n-3\rangle$ have at most two parts, by removing the staircase
$\langle 1\rangle$ (if possible) we can see that such partitions are enumerated by
$\binom{n+1}{2}_q$. 

From this, we immediately have that $F_{\lambda}$ decomposes as:
\begin{equation}\label{4}
F_{\lambda}(q)=1+q \binom{n+1}{2}_q+q^6\binom{n+1}{4}_q.
\end{equation}

We can assume for simplicity that $n\ge 8$, since the result is easy
to check (e.g., using Maple) for $n\le 7$. Let $c_i$ be the
coefficient of degree $i$ of $q^6\binom{n+1}{4}_q$. Hence, $c_i\neq 0$
if and only if $6\le i\le 4n-6$, and because of the unimodality of the
$q$-binomial coefficient $\binom{n+1}{4}_q$, the $c_i$ are also unimodal
with a peak at $c_{2n}$. Further, it easily follows from Lemma
\ref{fac} that $c_i>c_{i-1}$ for all $n+1\le i\le 2n$, with the
exception of $i=2n-1$, which gives $c_{2n-1}=c_{2n-2}$.  

On the other hand, notice that the $q$-binomial coefficient $q
\binom{n+1}{2}_q$ is a unimodal function; its coefficients $d_i$ are
nonzero for $1\le i\le 2n-1$, and they assume a peak at $d_n$. Also, it is
a simple exercise to check that when $\binom{n+1}{2}_q$ decreases
(or by symmetry, it increases),  it does so by at most 1. Finally,
notice that $d_{2n-1}=d_{2n-2}$ (they are both equal to 1), which
implies that the coefficients of $F_{\lambda}$ in degree $2n-1$ and
$2n-2$ are also equal.  

Putting all of the above together, since by equation~(\ref{4})
$F_{\lambda}$ can be written as 
$$F_{\lambda}(q)=1+\sum_{i=1}^{4n-6} (c_i+d_i)q^i,$$
it is easy to check that $F_{\lambda}$ is unimodal. (In fact, we have
shown that it has a  peak in degree $2n$.) 
\end{proof}

\begin{theorem}\label{at}
Set $\lambda=\langle n,n-t,n-2t,n-3t\rangle$, where $t\ge 2$ is
fixed. Then the rank-generating function $F_{\lambda}$ is nonunimodal
for all integers $n$ large enough with respect to $t$. 
\end{theorem}

\begin{proof}
Let $F_{\lambda}(q)=\sum_{i=0}^{4n-6t}c_i^{(t)}q^i$. We will prove that
$F_{\lambda}$ is nonunimodal for any  $t\ge 2$ and $n$ large enough
with respect to $t$, by showing that 
$$c_{2n}^{(t)}>c_{2n-1}^{(t)}<c_{2n-2}^{(t)}.$$

We assume from now on that $n$ is large enough. It is easy to see from
Lemma \ref{fac} and the proof of Theorem \ref{uni} that
$c_{2n-1}^{(1)}=c_{2n-2}^{(1)}$, and that $c_{2n}^{(1)}-c_{2n-1}^{(1)}$ goes to
infinity. Indeed, we have shown in that proof that, in degree $2n$,
the rank-generating function of $\langle n,n-1,n-2,n-3\rangle$ is the same as that
of $\binom{n+1}{4}_q$, whereas in degrees $2n-1$ and $2n-2$ it is
exactly one more than that of $\binom{n+1}{4}_q$, because of the extra
contributions coming from $q \binom{n+1}{2}_q$.

We begin by showing that $c_{2n}^{(2)}>c_{2n-1}^{(2)}<c_{2n-2}^{(2)}$.  Notice
that the only partition of $2n$ that lies inside
$\langle n,n-1,n-2,n-3\rangle$ but not inside $\langle n,n-2,n-4,n-6\rangle$ is
$\langle n,n-1,1\rangle$. Thus, $c_{2n}^{(2)}=c_{2n}^{(1)}-1$. Similarly, the only
partition of $2n-1$ that lies inside
$\langle n,n-1,n-2,n-3\rangle$ but not $\langle n,n-2,n-4,n-6\rangle$ is $\langle n,n-1\rangle$, and
therefore, $c_{2n-1}^{(2)}=c_{2n-1}^{(1)}-1$.

Finally, $c_{2n-2}^{(2)}=c_{2n-2}^{(1)}$, since all partitions with distinct
parts of $2n-2$ that lie inside $\langle n,n-1,n-2,n-3\rangle$ clearly cannot have
$n-1$ as their second largest part. It follows that the difference
between $(c_{2n}^{(1)},c_{2n-1}^{(1)},c_{2n-2}^{(1)})$ and
$(c_{2n}^{(2)},c_{2n-1}^{(2)},c_{2n-2}^{(2)})$ is $(1,1,0)$, which immediately
proves that $c_{2n}^{(2)}>c_{2n-1}^{(2)}<c_{2n-2}^{(2)}$, i.e., the theorem for
$t=2$.

In a similar fashion, it is easy to check that, in passing from
$\langle n,n-2,n-4,n-6\rangle$ to $\langle n,n-3,n-6,n-9\rangle$, the difference
between $(c_{2n}^{(2)},c_{2n-1}^{(2)},c_{2n-2}^{(2)})$ and
$(c_{2n}^{(3)},c_{2n-1}^{(3)},c_{2n-2}^{(3)})$ is $(3,2,2)$, showing the
result for $t=3$.

In general, if $t\ge 3$ and $n_i$ is the number of partitions of $i$ into two
distinct parts, employing the same idea as above easily gives us that
in passing from $\langle n,n-(t-1),n-2(t-1),n-3(t-1)\rangle $ to $\langle
n,n-t,n-2t,n-3t\rangle$, the difference $c_{2n}^{(t)}-c_{2n-1}^{(t)}$ decreases
by $n_{2t-3}-n_{t-2}$ with respect to $c_{2n}^{(t-1)}-c_{2n-1}^{(t-1)}$.

Notice that 
$$n_{2t-3}-n_{t-2}=\lceil (2t-3)/2\rceil - \lceil (t-2)/2\rceil = \lfloor t/2 \rfloor ,$$
if  as usual  we denote by $\lceil x\rceil$ and $\lfloor x \rfloor $ the smallest integer $\ge x$ and the largest  $\le x$, respectively.

Therefore, the difference between $c_{2n}^{(1)}-c_{2n-1}^{(1)}$ and $c_{2n}^{(t)}-c_{2n-1}^{(t)}$ amounts to $\sum_{i=1}^t\lfloor i/2\rfloor$, which has order of magnitude $t^2/4$. 

Since $c_{2n}^{(1)}-c_{2n-1}^{(1)}$ goes to infinity, this completes the proof that $c_{2n}^{(t)}>c_{2n-1}^{(t)}$ for all $t\ge 2$ and $n$  large enough with respect to $t$. 

Notice that $c_{2n-1}^{(2)}<c_{2n-2}^{(2)}$. Hence, in order to complete the
proof of the theorem it now suffices to show that, in passing from
$\langle n,n-(t-1),n-2(t-1),n-3(t-1)\rangle$ to $\langle n,n-t,n-2t,n-3t\rangle$,
$c_{2n-1}^{(t)}$ decreases at least as much as  $c_{2n-2}^{(t)}$ does, for any
$t\geq 3$. Thus, since  $n$ is large enough with respect to $t$, one
moment's thought gives that it is enough to show that there are at
least as  many partitions of $2n-1$ than there are
of $2n-2$, which are contained inside $\langle
n,n-(t-1),n-2(t-1),n-3(t-1)\rangle$ and have $n-(t-1)$ as their second
largest part.  

But if $\mu=\langle \mu_1,\mu_2=n-(t-1),\mu_3,\mu_4\rangle$ is such a partition of
$2n-2$, notice that $\mu_1\ge n-t+2$, and therefore
$$\mu_3+\mu_4\le 2n-2 - ( n-(t-1))- (n-t+2)=2t-5,$$
which is  smaller than $\mu_2=n-(t-1)$ by at least 2, since $n$ is large. 

Therefore, we can define an injection between the above partitions of
$2n-2$ and those of $2n-1$ by mapping $\mu$ to $\theta=\mu +
(0,0,1,0)$. This shows that there are at least as  many of the above
partitions of $2n-1$ as there are of $2n-2$, completing the
proof of the theorem. 
\end{proof}

\begin{remark}
The same idea of the proof of Theorem \ref{at} can prove the
nonunimodality of $F_{\lambda}$ also for other partitions
$\lambda=\langle \lambda_1, \lambda_2, \lambda_3,\lambda_4\rangle$;
namely, those that are obtained by ``perturbing'' $\langle
n,n-t,n-2t,n-3t\rangle$ in a way that the $\lambda_i$ remain within
suitable intervals. This fact, which is quite natural, is also
consistent with the results of Stanton \cite{stanton} in the case of
arbitrary partitions with parts lying within certain intervals. We
only remark here that in general, however, an actual ``interval
property'' (see e.g. \cite{Za10}) does not hold in this context. In
fact, it is easy to check, e.g. using Maple, that the rank-generating
function $F_{\lambda}$ is nonunimodal when $\lambda=\langle
19,16,11,8\rangle$ and $\lambda=\langle 19,16,9,8\rangle$, while it is
unimodal for $\lambda=\langle 19,16,10,8\rangle$.
\end{remark}

\section{Other shapes}

We begin by presenting a new $q$-analog of the binomial coefficients. For any integers $a$ and $b$ such that
$1\le b\le a/2$, define  $F_{\lambda}=\binom{a}{b}^q$ to be the
rank-generating function of the partitions with distinct
parts contained inside $\lambda=\langle
a-1,a-3,\dots,a-(2b-1)\rangle$. In Proposition \ref{ab}, we will show that $\binom{a}{b}^q$ is a $q$-analog of the binomial coefficient $\binom{a}{b}$. After discovering an independent proof
of this fact, we found out that it can also be easily deduced from a
theorem of R. Proctor concerning shifted plane partitions (see
\cite[Theorem 1]{Pr2}). However, since our argument, unlike Proctor's,
is combinatorial, we include a sketch of it below for completeness. We
use the following lemma without proof, since it is equivalent to the
well-known fact that the number of standard Young tableaux with $a$
boxes, at most two rows, and at most $b$ boxes in the second row is
$\binom ab$.

\begin{lemma}\label{binary}
Fix integers $a$ and $b$ such that $1\le b\le a/2$. Then the number of
binary sequences of length $a$ containing at most $b$ 1's, and such
that no initial string  contains more 1's than 0's, is
$\binom{a}{b}$. 
\end{lemma}

\begin{proposition}\label{ab}
Fix integers $a$ and $b$ such that $1\le b\le a/2$, and let
$\lambda=\langle a-1,a-3,\dots,a-(2b-1)\rangle$. Then the number of
partitions with distinct parts lying inside $\lambda$ is
$\binom{a}{b}$.
\end{proposition}

\begin{proof} Let $\mu=\langle
  \mu_1,\dots,\mu_t\rangle$ be 
  a partition with distinct parts contained inside
  $\lambda=\langle a-1,\dots,a-(2b-1)\rangle$. In particular,  $a-1\ge
  \mu_1>\dots >\mu_t\ge 1$, where $t\le b$. 

We associate to $\mu$ a binary sequence of length $a$, say
$W_{\mu}=w_aw_{a-1}\cdots w_1$,  such  that $w_i=1$ if $i$ is a part
of $\mu$, and $w_i=0$ otherwise. (Notice that $w_a$ is always 0, since
$a>\mu_1$ for all partitions $\mu$.)  

By definition, the number of 1's in $W_{\mu}$ is  $t\le b$, and since
the parts of $\lambda$ differ by exactly 2, it is a standard exercise
to check  that the above correspondence is indeed a bijection between
our partitions $\mu$ and those binary sequences $W_{\mu}$ where no
initial string of $W_{\mu}$ contains more 1's than 0's. Thus, by Lemma
\ref{binary}, the number of partitions $\mu$ is $\binom{a}{b}$, as
desired. 
\end{proof}

%Therefore, from Proposition \ref{ab} we immediately have the following result. 

%\begin{corollary}\label{qb}
%If $\lambda=\langle a-1,a-3,\dots,a-(2b-1)\rangle$,  then $F_{\lambda}(q)=\binom{a}{b}^q$. 
%\end{corollary}

For example,
$$F_{\langle 4,2\rangle}=\binom{5}{2}^q=1+q+q^2+2q^3+2q^4+2q^5+q^6;$$
$$F_{\langle 5,3\rangle}=\binom{6}{2}^q=1+q+q^2+2q^3+2q^4+3q^5+2q^6+2q^7+q^8;$$
$$F_{\langle 8,6,4,2\rangle}=\binom{9}{4}^q=1+q+q^2+2q^3+2q^4+3q^5+4q^6+5q^{7}+6q^{8}+7q^{9}+8q^{10}+$$$$9q^{11}+10q^{12}+11q^{13}+12q^{14}+11q^{15}+11q^{16}+10q^{17}+7q^{18}+4q^{19}+q^{20}.$$

Notice that, for $b>1$, our $q$-analog $\binom{a}{b}^q$ of the binomial
coefficient $\binom{a}{b}$ is always different from the $q$-binomial
coefficient $\binom{a}{b}_q$ (in fact, it is easy to show
that $\binom{a}{b}^q$ is never  symmetric  for $b>1$). While
$\binom{a}{b}_q$ is well known to be unimodal, the case $t=2$ of
Theorem \ref{at} proves that, in general, unimodality may fail quite
badly for $\binom{a}{b}^q$, even when $b=4$. The smallest such
nonunimodal example is
$$F_{\langle 9,7,5,3\rangle}=\binom{10}{4}^q=1+q+q^2+2q^3+2q^4+3q^5+4q^6+5q^{7}+6q^{8}+8q^{9}+9q^{10}+10q^{11}+12q^{12}+13q^{13}$$$$+15q^{14}+16q^{15}+\textbf{17}q^{16}+\textbf{16}q^{17}+\textbf{17}q^{18}+15q^{19}+14q^{20}+11q^{21}+7q^{22}+4q^{23}+q^{24}.$$

However, we conjecture that the following fact is true, which is a
special case of a conjecture that we will state later.

\begin{conjecture}\label{centr}
Our $q$-analogs of the \emph{central} binomial coefficients are all 
unimodal. That is,  $\binom{2n+1}{n}^q$ and $\binom{2n}{n}^q$ are both
unimodal, for any $n\ge 1$. 
\end{conjecture}

%We have verified Conjecture \ref{centr} on Maple for all $n\le .............$.
For partitions with three parts, all rank-generating functions are
unimodal. We will provide a bijective proof of this result, assuming
the corresponding theorem of Stanton for arbitrary partitions
(\cite[Theorem 7]{stanton}). 

\begin{lemma}[Stanton]\label{3} % \emph{Stanton}
If $\lambda$ is any arbitrary partition of length $b\le 3$, then the
rank-generating function $G_{\lambda}$ is unimodal. 
\end{lemma}

\begin{lemma}\label{p-1} 
Consider the partition $\lambda=(p,r,s)$, and let
$G_{(p,r,s)}(q)=\sum_{i=0}^{p+r+s}a_iq^i$. We have: 
\begin{enumerate}
\item If $2\le p\le 2r+s$, then $a_{p-1} <a_p$:
\item If $p=1$ or $p\ge 2r+s+1$, then $a_{p-1} =a_p$ and $a_i\ge
  a_{i+1}$ for all $i\ge p-1$. 
\end{enumerate}
\end{lemma}

\begin{proof}
Since the largest part of $\lambda$ is $p$, notice that there is a
natural injection $\phi$ from the set $A_{p-1}$ of partitions $\mu$
of $p-1$ contained inside the Ferrers diagram of $\lambda$ to the set $A_p$, where $\theta=\phi(\mu)=\mu+(1,0,0)$.  

Thus, for any $\lambda=(p,r,s)$, we have $a_{p-1} \le
a_p$. Clearly, equality holds if and only if there exists no partition
$\theta=(\theta_1,\theta_2,\theta_3)$ in $A_p$ such that
$\theta_1=\theta_2$, since these are the only partitions not in the
image of the map $\phi$.

It is a standard exercise now to show that if $2\le p\le 2r+s$, then there
always exists a partition $\theta \in A_p$ such that
$\theta_1=\theta_2$. Indeed, if $p=2r+s$, then we can pick
$\theta=(r,r,s)$; for $2\le p<2r+s$, one can for instance first decrease
the value of $s$ until it reaches 0 (i.e., until $p$ is down to $2r$),
and then consider the partitions $\theta=(d,d,\epsilon)$, where $d$
decreases by 1 at the time and $\epsilon$ is either 0 or 1, depending
on the parity of $p$. This proves part (1). 

In order to prove (2), notice that the case $p=1$ is trivial, since
here  $\lambda=(p,r,s)=(1,1,1)$. Thus let $p\ge 2r+s+1$. Then we have
that no partition $\theta$ of $p$ inside $\lambda=(p,r,s)$ can satisfy
$\theta_1=\theta_2$, since $\theta_3\le s$, $\theta_2\le r$, and
therefore $\theta_1\ge r+1$. Thus, this is exactly the case where
$a_{p-1} =a_p$, and in order to finish the proof of the lemma, now it
suffices to show that $a_i\ge a_{i+1}$ for all $i\ge p-1$.  

But this can be done in a symmetric fashion to the above argument, by
defining a map $\psi$ from $A_{i+1}$ to $A_i$ such that
$\beta=\psi(\alpha)=\alpha-(1,0,0)$. Since  $p\ge 2r+s+1$, it is easy
to see that $\psi$ is well defined and injective. Thus, $a_i\ge
a_{i+1}$ for all $i\ge p-1$, as we wanted to show. 
\end{proof}

\begin{theorem}
If $\lambda$ is any partition with length $b\le 3$, then the
rank-generating function $F_{\lambda}$ is unimodal. 
\end{theorem}

\begin{proof} When $b=1$ the result is obvious, and when $b=2$ it is
  also easy to check. Indeed, this can be done directly, or by
  observing that if $\lambda=\langle p+1,r\rangle$, for some $p\ge r\ge 1$, then
  one promptly obtains that 
$$F_{\langle p+1,r\rangle}(q)=1+qG_{(p,r)}(q).$$

Thus $F_{\lambda}$ is unimodal, since $G_{(p,r)}(q)$ is unimodal by
Lemma \ref{3}. 

Hence, let $b=3$, and set $\lambda=\langle p+2,r+1,s\rangle$, where $p\ge r\ge s\ge
1$. Clearly, any partition $\mu$ contained inside $\lambda$ has at
most three parts, and it is easy to see that those with at least two
parts are in bijective correspondence with arbitrary partitions
contained inside $(p,r,s)$, by removing the staircase $\langle 2,1\rangle $. 

From this, it  follows that
$$F_{\langle p+2,r+1,s\rangle}(q)=q^3G_{(p,r,s)}(q)+(1+q+q^2+\dots +q^{p+2}).$$

Therefore, since by Lemma \ref{3}, $G_{(p,r,s)}(q)$ is unimodal, we
have that in order to prove the unimodality of $F_{\langle p+2,r+1,s\rangle}(q)$,
it suffices to show that if the coefficients of $G_{(p,r,s)}(q)$
coincide in degree $p-1$ and $p$, then they are nonincreasing from
degree $p-1$ on. But this follows from Lemma \ref{p-1}, thus
completing the proof of the theorem. 
\end{proof}

For partitions $\lambda$ with $b\ge 5$ parts, the scenario becomes
more and more unclear, and it again bears several similarities with
Stanton's situation for arbitrary partitions. For instance, when
$b=5$, all examples we have computed are unimodal, and for $b=6$,
while it is possible to construct nonunimodal partitions, we have not
been able to place them into any infinite family.  

In  particular, even the ``truncated staircases'' $\lambda=\langle n,n-1,
\dots,n-(b-1)\rangle$ in general need not be unimodal when $b<n$. For
instance, the rank-generating function $F_{\lambda}$ is nonunimodal
for $\lambda=\langle 15,14,13,12,11,10\rangle$, $\lambda=\langle 17,16,15,14,13,12\rangle$, and
$\lambda=\langle 19,18,17,16,15,14\rangle$, though this sequence does not continue
in the obvious way. In fact, L. Alpoge \cite{alp} has recently proved, by means of a nice analytic argument, that the truncated staircases are all unimodal for $n$ sufficiently large with respect to $b$, a fact that was conjectured in a previous version of our paper.

%when $\lambda=\langle n,n-1, \dots,n-(b-1)\rangle$, it is not difficult to see that $F_{\lambda}$ can be decomposed as a
%shifted sum of suitable $q$-binomial coefficients, with a similar argument
%to the one we used in the proof of Theorem \ref{at} for $b=4$. Thus we conjecture the following.

%\begin{theorem}[\cite{alp}]
%Let $\lambda=\langle n,n-1, \dots,n-(b-1)\rangle$. Then the rank-generating
%function $F_{\lambda}$ is always unimodal for $n$  large enough with respect to $b$. 
%\end{theorem}

%Of course, the results of this paper imply that Conjecture \ref{large} is true for $b\le 4$. 

As we mentioned earlier, recall that for $b=n$, i.e., for the
staircase partitions $\lambda=\langle b,b-1, \dots,2,1\rangle$, the unimodality of
$F_{\lambda}$ has already been established, though no combinatorial
proof is known to date. The following conjecture attempts to place
this result into a much broader context.

\begin{conjecture}\label{arith}
The rank-generating function $F_{\lambda}$ is unimodal for all
partitions $\lambda=\langle a, a-t, ..., a-(b-1)t\rangle$ such that $t \ge a/b$. In
other words, unimodality holds for all partitions with parts in
arithmetic progression that begin with the smallest possible positive
integer. 
\end{conjecture}

%We have verified Conjecture \ref{arith} on Maple for ...................
Notice that, again similarly to Stanton's situation of arbitrary
partitions, all examples we have constructed of nonunimodal
rank-generating functions $F_{\lambda}$  have exactly two
peaks. However, it seems reasonable to expect that nonunimodality may
occur with an arbitrary number of peaks, though showing  this fact
will probably require a significantly new idea. 

\begin{conjecture}\label{arith2}
For any integer $N\ge 2$, there exists a partition $\lambda$ whose
rank-generating function $F_{\lambda}$ is nonunimodal with (exactly?)
$N$ peaks. 
\end{conjecture}

Finally, recall that a polynomial $\sum_{i=0}^N a_iq^i$ is
\emph{flawless} if $a_i\le a_{N-i}$ for all $i\le N/2$. Though this
property is not as well studied as symmetry or unimodality, several
natural and important sequences in algebra and combinatorics happen to
be flawless (see e.g. \cite{BMMNZ, BMMNZ2, Hi}). We conclude this section
by stating the following intriguing conjecture.
 
\begin{conjecture}\label{arith3}
For any partition $\lambda$, the rank-generating function
$F_{\lambda}$ is flawless. 
\end{conjecture}

\section{Acknowledgements} The second author warmly thanks the first
author for his hospitality during calendar year 2013 and the MIT Math
Department for partial financial support. The two authors wish to thank the referees for several comments that improved the presentation of this paper. They also want to acknowledge the use of the computer package Maple, which has been of
invaluable help in suggesting the statements of some of the results and conjectures.

%-------------------------------------------

\end{document}